\documentclass{amsart}

\usepackage{amssymb}
\usepackage[utf8]{inputenc}
\usepackage[T1]{fontenc}
\usepackage[english]{babel}
\usepackage{xargs}
\usepackage{enumitem}
\usepackage{hyperref}
\hypersetup{
	hidelinks,
  colorlinks  = true,    
  urlcolor    = blue,    
  linkcolor   = blue,    
  citecolor   = blue,      
  pdfauthor   = {Antonio Montero and Asia Weiss},%
  pdfsubject  = {Regular hypertopes},%
  pdftitle    = {Locally spherical hypertopes from generalized cubes},  %
}

\usepackage[capitalise,noabbrev, nameinlink]{cleveref}
\usepackage{tikz-cd}
\theoremstyle{plain}
\newtheorem{thm}{Theorem}[section]
\newtheorem{theorem}[thm]{Theorem}
\newtheorem{lemma}[thm]{Lemma}

\newtheorem{proposition}[thm]{Proposition}

\newtheorem{corollary}[thm]{Corollary}

\theoremstyle{definition}

\theoremstyle{remark}

\newtheorem{rem}[thm]{Remark}

\numberwithin{equation}{section}

\renewcommand{\leq}{\leqslant} %
\renewcommand{\geq}{\geqslant}
\renewcommand{\epsilon}{\varepsilon} %
\renewcommand{\subset}{\subseteq} %

\renewcommand{\{}{\lbrace}
\renewcommand{\}}{\rbrace}
\newcommand{\sm}{\setminus} 

\newcommand{\cP}{\mathcal{P}}
\newcommand{\cF}{\mathcal{F}}

\newcommand{\cK}{\mathcal{K}}
\newcommand{\cH}{\mathcal{H}}

\newcommand{\bZ}{\mathbb{Z}}
\newcommand{\bN}{\mathbb{N}}

\newcommand{\cyvec}[1]{\bar{\mathrm{#1}}}
\newcommand{\vx}{\cyvec{x}}
\newcommand{\vy}{\cyvec{y}}

\newcommand{\rr}{\varrho}

\DeclareMathOperator{\aut}{Aut} 
\DeclareMathOperator{\stab}{Stab}
\DeclareMathOperator{\rk}{rk}
\DeclareMathOperator{\Hal}{H}
\DeclareMathOperator{\Hyp}{\cH}
\DeclareMathOperator{\semi}{\mathcal{S}}

\newcommand{\hypP}[1][\cP]{\Hyp(#1)}
\newcommand{\autI}{\aut_{I}}
\newcommand{\halP}[1][\cP]{\Hal(#1)}

\newcommand{\twoK}[1][\cK]{2^{#1}}
\newcommand{\dtwoK}[1][\cK]{\hat{2}^{#1}}

\begin{document}


\title{Locally spherical hypertopes from generlised cubes}
\thanks{This research was supported by NSERC}

\author[A. Montero]{Antonio Montero}
\address[A. Montero]{Department of Mathematics and Statistics, York University, Toronto, Ontario M3J 1P3, Canada}
\email[A. Montero]{amontero@yorku.ca}

\author[A.I. Weiss]{Asia Ivi\'c Weiss}
\address[A.I. Weiss]{Department of Mathematics and Statistics, York University, Toronto, Ontario M3J 1P3, Canada}
\email[A.I. Weiss]{weiss@mathstat.yorku.ca}




\maketitle
\begin{abstract}
	We show that every non-degenerate regular polytope can be used to construct a thin, residually-connected, chamber-transitive incidence geometry, i.e. a regular hypertope, with a tail-triangle Coxeter diagram. 
	We discuss several interesting examples derived when this construction is applied to generalised cubes. In particular, we produce an example of a rank $5$ finite locally spherical  proper hypertope of hyperbolic type. 
\end{abstract}




\section{Introduction}

Hypertopes are  particular kind of incidence geometries that generalise the notions of abstract polytopes and of hypermaps. 
The concept was introduced in \cite{FernandesLeemansWeiss_2016_HighlySymmetricHypertopes} with particular emphasis on regular hypertopes (that is, the ones with highest degree of symmetry). 
Although in \cite{FernandesLeemansPiedadeWeiss_TwoFamiliesLocally_preprint, FernandesLeemansWeiss_ExplorationLocallySpherical_preprint, FernandesLeemansWeiss_2018_HexagonalExtensionsToroidal} a number of interesting examples had been constructed, within the theory of abstract regular polytopes much more work has been done. Notably,  \cite{Schulte_1983_ArrangingRegularIncidence} and \cite{Schulte_1985_ExtensionsRegularComplexes} deal with the universal constructions of polytopes, while in \cite{Danzer_1984_RegularIncidenceComplexes, Pellicer_2009_ExtensionsRegularPolytopes, Pellicer_2010_ExtensionsDuallyBipartite} the constructions with prescribed combinatorial conditions are explored. 
In another direction, in \cite{ CameronFernandesLeemansMixer_2017_HighestRankPolytope,  FernandesLeemans_2018_CGroupsHigh, LeemansMoerenhoutOReillyRegueiro_2017_ProjectiveLinearGroups, Pellicer_2008_CprGraphsRegular} the questions of existence of polytopes with prescribed (interesting) groups are investigated. Much of the impetus to the development of the theory of abstract polytopes, as well as the inspiration with the choice of problems, was based on work of Branko Gr\"unbaum \cite{
Gruenbaum_1978_RegularityGraphsComplexes} from 1970s.

In this paper we generalise the halving operation on polyhedra (see 7B in  \cite{McMullenSchulte_2002_AbstractRegularPolytopes}) on a certain class of regular abstract polytopes to construct regular hypertopes.  
More precisely, given a regular non-degenerate $n$-polytope $\cP$, we construct a regular hypertope $\hypP$ related to semi-regular polytopes with tail-triangle Coxeter diagram.

The paper is organised as follows.
In \cref{sec:hypertopes} we review the basic theory of hypertopes (with particular focus on regular hypertopes) and
revisit the notion of a regular polytope (first introduced in early 1980s) within the theory of hypertopes. 
In \cref{sec:halving} we explore the {halving} operation on an abstract polytope and show that the resulting incidence system is a regular hypertope.
Finally, in \cref{sec:lshGeneralisedCubes} we give concrete examples arising from our construction. 
In particular, we focus on locally spherical hypertopes arising from Danzer's construction of generalised cubes.
As a result we produce an example of a finite regular rank $5$ proper hypertope of hyperbolic type.

\section{Regular hypertopes} \label{sec:hypertopes}

%
In this section we review the definition and basic properties of regular hypertopes. 
We introduce abstract polytopes as a special class of hypertopes. 
However, if the reader is interested in a classic and more detailed definition of abstract polytopes, we suggest \cite[Section 2A]{McMullenSchulte_2002_AbstractRegularPolytopes}.

The notion of a regular hypertope was introduced in \cite{FernandesLeemansWeiss_2016_HighlySymmetricHypertopes} as a common generalisation of abstract regular polytopes and hypermaps. 
In short, a \emph{regular hypertope} is a \emph{thin}, \emph{residually-connected}, \emph{chamber-transitive} geometry (the concepts are defined below). 
More details and an account of general theory can be found in \cite{BuekenhoutCohen_2013_DiagramGeometry}.

An \emph{incidence system} is a $4$-tuple $\Gamma:=(X,\ast,t,I)$ satisfying the following conditions: 
\begin{itemize}
 \item $X$ is the set of elements of $\Gamma$;
 \item $I$ is the set of \emph{types of $\Gamma$} (whose cardinality is called the \emph{rank of $\Gamma$});
 \item $t : X\rightarrow I$ is a type function, associating to each element $x\in X$ a type $t(x)\in I$ (an element $x$ is said to be of \emph{type} $i$, or an \emph{$i$-element}, whenever $t(x)=i$, for $i\in I$); and
 \item $\ast$ is a binary relation in $X$ called \emph{incidence}, which is reflexive, symmetric and such that, for all $x,y\in X$,
if $x\ast y$ and $t(x)=t(y)$, then $x=y$.
\end{itemize}

The \emph{incidence graph} of $\Gamma$ is a graph whose vertices are the elements of $X$ and where  two vertices are connected whenever they are incident in $\Gamma$.
The type function determines an $|I|$-partition on the set of elements of the incidence graph. 

A \emph{flag} $F$ is a subset of $X$ in which every two elements are incident. An element $x$ is incident to a flag $F$, denoted by $x\ast F$, when $x$ is incident to all elements of $F$.
For a flag $F$ the set $t(F):=\{t(x)\,|\, x\in F\}$ is called \emph{the type of $F$}. 
When $t(F)=I$, $F$ is called a \emph{chamber}.

An incidence system $\Gamma$ is a \emph{geometry} (or an \emph{incidence geometry}) if every flag of $\Gamma$ is contained in a chamber, that is, if all maximal flags of $\Gamma$ are chambers.

The \emph{residue} of a flag $F$ of an incidence geometry $\Gamma$ is the
incidence geometry $\Gamma_F := (X_F,\ast_F,t_F,I_F)$ where $X_F := \{x\in X: x\ast F , x\notin F\}$, 
$I_F:= I\setminus t(F)$, and where $t_F$ and $\ast_F$ are restrictions of $t$ and $\ast$ to $X_F$ and $I_F$ respectively.

An incidence system $\Gamma$ is  \emph{thin} when every residue of rank one of $\Gamma$ contains exactly two elements.   
If an incidence geometry is thin, then given a chamber $C$ there exists exactly one chamber differing from $C$ in its $i$-element which we denote by $C^i$. We  also say that $C$ and $C^i$ are  \emph{$i$-adjacent}. 
An incidence system $\Gamma$ is \emph{connected} if its incidence graph is connected. 
Moreover, $\Gamma$ is \emph{residually connected}  when $\Gamma$ is connected and each residue of $\Gamma$ of rank at least two is also connected. 
It is easy to see that this condition is equivalent to \emph{strong connectivity} for polytopes (as defined in \cite[p. 23]{McMullenSchulte_2002_AbstractRegularPolytopes} and reviewed below) and the thinness is equivalent to to the diamond condition for polytopes.
A \emph{hypertope} is a thin incidence geometry which is residually connected.

An \emph{abstract polytope} of rank $n$ is usually defined as a strongly-connected partially ordered set $(\cP, \leq)$ such that $\cP$ has a maximum and a minimum element. 
We also require that $\cP$ satisfies the \emph{diamond condition} and in such a way that all maximal chains of $\cP$ have the same length ($n+2$). 
In the language of incidence geometries, an abstract polytope is an incidence system $(\cP, \ast_{\leq}, \rk, \left\{ -1, \dots, n \right\} )$, where $\ast_{\leq}$ is the incidence relation defined by the order of $\cP$ (i.e., $x \ast_{\leq} y$ if and only if $x \leq y$ or $y \leq x$) and $\rk$ is the \emph{rank function}. 
We require that $\cP$ has a unique (minimum) element of rank (type) $-1$  and a unique (maximum) element of rank $n$.
Note that a flag (in the language of incidence systems) is what has been called a \emph{chain} in the theory of abstract polytopes.
Therefore, maximal chains of $\cP$ are precisely the chambers of the corresponding incidence system.
The fact that every maximal chain of $\cP$ has $(n+2)$ elements implies that $\cP$ defines a geometry.
It is well-known that for any two incident elements $F_{i} \leq F_{j}$ of $\cP$, with $\rk(F_{i}) = i$ and $\rk(F_{j}) = j$, the \emph{section}  $F_{j}/F_{i} = \{x \in \cP : F_{i} \leq x \leq F_{j}\}$ is a $(j-i-1)$-polytope.
We note that for polytopes, the residue of a chain $F$ is a union of {sections} of $\cP$ defined by the intervals of $I_{F}$.

Observe that rank $2$ hypertopes are precisely the abstract polygons and rank $3$ hypertopes are non-degenerate hypermaps.

A \emph{type-preserving automorphism} of an incidence system $\Gamma:=(X,\ast,t,I)$ is a permutation $\alpha$ of $X$ such that for every $x \in X$, $t(x) = t(x \alpha)$  and if $x,y \in X$, then $x\ast y$ if and only if $x \alpha \ast y\alpha$. 
The set of type-preserving automorphisms of $\Gamma$ is denoted by $\autI(\Gamma)$.

The group of type-preserving automorphisms of an incidence geometry $\Gamma$ generalises the automorphism group of an abstract polytope. 
Some familiar symmetry properties on polytopes extend naturally to incidence geometries. 
For instance, $\autI(\Gamma)$ acts faithfully on the set of its chambers.
Moreover, if $\Gamma$ is a hypertope this action is semi-regular. 
In fact, if $\alpha \in Aut_I(\Gamma)$ fixes a chamber $C$, it also fixes its $i$-adjacent chamber $C^i$. 
Since $\Gamma$ is residually connected, $\alpha$ must be the identity.  

We say that $\Gamma$ is \emph{chamber-transitive} if the action of $Aut_I(\Gamma)$ on the chambers is transitive, and in that case the action of $\Gamma$ on the set of chambers is regular. 
For that reason $\Gamma$ is then called a  \emph{regular hypertope}. As expected, this generalises the concept of a regular polytope.

Observe that, when $\Gamma$ is a geometry, chamber-transitivity is equivalent to \emph{flag-transitivity} (meaning that for each $J\subseteq I$, there is a unique orbit on the flags of type $J$ under the action of $Aut_I(\Gamma)$; see for example Proposition 2.2 in  \cite{FernandesLeemansWeiss_2016_HighlySymmetricHypertopes}).

Let $\Gamma:=(X,\ast,t,I)$ be a regular hypertope and let $C$  be a fixed (base) chamber of $\Gamma$. 
For each $i\in I$  there exists exactly one automorphism $\rho_i$ mapping $C$ to $C^i$.
If $F \subset C$ is a flag, then the automorphism group of the residue $\Gamma_{F}$ is precisely stabiliser of $F$ under the action of $\autI(\Gamma)$. 
We denote this group by $\stab_{\Gamma}(F)$.
It is easy to see that \[\stab_{\Gamma}(F) = \left\langle \rho_{i} : i \in I_{F} \right\rangle.  \]
If $I_{F}=\left\{ i \right\} $, that is  $\Gamma_{F}$ is of rank $|I|-1$, the thinness of $\Gamma$ imply that
\begin{equation} \label{eq:rho_involutions}
 \rho_{i}^{2} = 1.
\end{equation}
If $I_{F}=\left\{ i,j \right\} $, then there exists $p_{ij} \in \left\{ 2, \dots, \infty \right\} $ such that
\begin{equation} \label{eq:rho_Coxeter}
 \left(\rho_{i}\rho_{j}\right)^{p_{ij}} = 1.
\end{equation}
Moreover, if $F,G \subset C$ are flags such that $I_{F} = J$ and $I_{G} = K$, then 
\[\stab_{\Gamma}(F) \cap \stab_{\Gamma}(G) = \stab_{\Gamma}(F \cup G),\] or equivalently
\begin{equation}\label{eq:intCondition}
 \left\langle \rho_{j} : j \in J  \right\rangle \cap \left\langle \rho_{k} :k \in K \right\rangle = \left\langle \rho_{i} : i \in J\cap K \right\rangle.   
\end{equation}

We call the condition in \eqref{eq:intCondition} the \emph{intersection condition}.
Following \cite{FernandesLeemansWeiss_2016_HighlySymmetricHypertopes}, a C-\emph{group}  is a group generated by involutions $\left\{  \rho_{i} : i \in I  \right\}$ that satisfies the intersection condition. 
It follows that the type-preserving automorphism group of a regular hypertope is a C-group (\cite[Theorem 4.1]{FernandesLeemansWeiss_2016_HighlySymmetricHypertopes}).

Every Coxeter group $U$ is a C-group and in particular, it is the type-preserving automorphism of a regular hypertope \cite[Section 3]{Tits_1961_GroupesEtGeometries} called the \emph{universal regular hypertope} associated with the Coxeter group $U$. 
Moreover, every C-group $G$ is a quotient of a Coxeter group $U$. 
If $\cH$ is a regular hypertope whose type-preserving automorphism group is $G$, the \emph{universal cover} of $\cH$ is the regular hypertope associated with $U$. 

The \emph{Coxeter diagram} of a C-group $G$ is a graph with $|I|$ vertices corresponding to the generators of $G$ and with an edge $\{i,j\}$ whenever the order $p_{ij}$ of  $\rho_i\rho_j$ is greater than $2$. The edge is endowed with the label $p_{ij}$ when $p_{ij}>3$. 
The automorphism group of an abstract polytope  is a \emph{string C-group}, that is, a C-group having a linear Coxeter diagram. If $\cP$ is a regular $n$-polytope, then we say that $\cP$ is of (Schläfli) type $\{p_{1}, \dots,p _{n-1}\}$ whenever the Coxeter diagram of $\aut(\cP)$ is 
\[
	\begin{tikzcd}[row sep=tiny, start anchor = center, end anchor = center]
		\overset{\rho_{0}}{\bullet} \ar[r, dash, "p_{1}"' ] &%
		\overset{\rho_{2}}{\bullet} \ar[r, dash, "{\dots}" description] &%
		\overset{\rho_{n-2}}{\bullet} \ar[r, dash, "p_{n-1}"']&%
		\overset{\rho_{n-1}}{\bullet}
  \end{tikzcd}
\]

One of the most remarkable results in the theory of abstract regular polytopes is that the string C-groups are precisely the automorphism groups of the regular polytopes.
In other words, given a string C-group $G$, there exists a regular polytope $\cP=\cP(G)$ such that $G = \aut(\cP)$ (see \cite[Section 2E]{McMullenSchulte_2002_AbstractRegularPolytopes}). This result was first proved in \cite{Schulte_1980_RegulareInzidenzkomplexe_phdThesis, Schulte_1983_RegulareInzidenzkomplexe.Ii} for so-called \emph{regular incidence complexes}, (combinatorial objects slightly more general than abstract polytopes). See \cite[Section 8]{Schulte_2018_RegularIncidenceComplexes} for some historical notes on the subject.

Analogously, it is also possible to construct, under certain conditions, a regular hypertope from a group, and particularly from a C-group, using the following proposition.

\begin{proposition}[Tits Algorithm \cite{Tits_1961_GroupesEtGeometries}] \label{prop:titsAlgorithm}
 Let $n$ be a positive integer and $I:=\{0,\ldots,n-1\}$. 
 Let $G$ be a group together with a family of subgroups $(G_i)_{i\in I}$, $X$
 the set consisting of all cosets $G_ig$ with $g\in G$ and $i\in I$, and $t:X\rightarrow I$ defined by $t(G_ig)=i$. 
 Define an incidence relation $\ast$ on $X\times X$ by: 
 \[\begin{aligned}
 G_ig_1 &\ast G_jg_2 &&\text{ if and only if }& G_ig_1\cap G_jg_2 &\neq \emptyset .
 \end{aligned}\]
 Then the $4$-tuple $\Gamma:=(X,\ast, t,I)$ is an incidence system having $\{G_{i} : i \in I\}$ as a chamber.
 Moreover, the group $G$ acts by right multiplication as an automorphism group on $\Gamma$.
 Finally, the group $G$ is transitive on the flags of rank less than $3$.
\end{proposition}

The incidence system constructed using the proposition above  will be denoted by $\Gamma(G; (G_i)_{i\in I})$ and called a \emph{coset incidence system}.

\begin{theorem}[{\cite[Theorem 4.6]{FernandesLeemansWeiss_2016_HighlySymmetricHypertopes}}]\label{thm:CGroupFlagTrans_Hypertope}
Let $I= \{0, \dots, n-1\}$, let $G=\langle \rho_i\,|\,i\in I\rangle$ be a C-group, and let $\Gamma := \Gamma(G;(G_i)_{i\in I})$ where $G_i:=\langle \rho_j\,|\,j\neq i\rangle$ for all $i\in I$.
If $G$ is flag-transitive on $\Gamma$, then $\Gamma$ is a regular hypertope.
\end{theorem}

In other words, the coset incidence system $\Gamma = \Gamma(G,(G_{i})_{i \in I})$ is a regular hypertope if and only if the group $G$ is a C-group and $\Gamma$ is flag-transitive. 
In order to prove that a given group $G$ is a C-group, we can use the following result.

\begin{proposition}[{\cite[Proposition 6.1]{FernandesLeemans_2018_CGroupsHigh}}]\label{prop:IC} Let $G$ be a group generated by $n$ involutions  $\rho_0,\ldots, \rho_{n-1}$. Suppose that $G_i$ is a C-group for every $i\in\{0,\ldots,n-1\}$. Then $G$ is a C-group
if and only if $G_i \cap G_j = G_{i,j}$ for all $0\leq i,\,j \leq n-1$.
\end{proposition}

At the end of this section, we introduce the lemma that will be used in \cref{sec:halving} to prove our main results.

\begin{lemma}\label{lemma:CgroupsProduct}
	Let $G=\left\langle \rho_{0}, \dots, \rho_{r-1} \right\rangle $ and $H = \left\langle \rho_{r}, \dots, \rho_{r+s-1}  \right\rangle $ be two C-groups. Then the group \[G \times H = \left\langle \rho_{0}, \dots, \rho_{r-1}, \rho_{r}, \dots, \rho_{r+s-1} \right\rangle \] is a C-group.
\end{lemma} 
\begin{proof}
Let $I, J \subset \left\{ 0, \dots, r+s-1 \right\} $ and let $I_{1}, J_{1} \subset \left\{ 0, \dots ,r-1 \right\} $ and $I_{2}, J_{2}\subset\left\{ r, \dots, r+s-1 \right\} $ be such that $I = I_{1} \cup I_{2}$ and $J = J_{1} \cup J_{2} $. Let $\gamma \in \left\langle \rho_{i} : i \in I \right\rangle  \cap \left\langle \rho_{j}: j \in J \right\rangle $. 
Since $\gamma \in \left\langle \rho_{i}: i \in I \right\rangle $, there exists $\alpha_{1} \in \left\langle \rho_{i} : i \in I_{1} \right\rangle $ and $\alpha_{2} \in \left\langle \rho_{i} : i \in I_{2} \right\rangle $ such that $\gamma = \alpha_{1} \alpha_{2}$.
Similarly, there exists $\beta_{1} \in \left\langle \rho_{j} : j \in J_{1} \right\rangle $ and $\beta_{2} \in \left\langle \rho_{j} : j \in J_{2} \right\rangle $ such that $\gamma = \beta_{1} \beta_{2}$. 
It follows that \[\beta_{1}^{-1} \alpha_{1} =  \beta_{2}\alpha_{2}^{-1}.\]
But $\beta_{1}^{-1} \alpha_{1} \in G$,  $\beta_{2}\alpha_{2}^{-1} \in H$,and since $G\cap H = \left\{ 1 \right\} $, it follows that $\alpha_{1} = \beta_{1}$. 
This implies \[\alpha_{1} \in \left\langle \rho_{i} : i \in I_{1} \right\rangle \cap \left\langle \rho_{j} : j \in J_{1} \right\rangle = \left\langle \rho_{k} : k \in I_{1} \cap J_{1} \right\rangle, \]
where the last equality follows from the fact that $G$ is a C-group.
Similarly, we can conclude that $\alpha_{2} \in \left\langle \rho_{k} : k \in I_{2} \cap J_{2} \right\rangle$ which implies that \[\gamma = \alpha_{1} \alpha_{2} \in \left\langle \rho_{k} : k \in \left( I_{1} \cap J_{1} \right) \cup \left( I_{2} \cap J_{2} \right) \right\rangle .\]
Finally, just observe that $I \cap J = \left( I_{1} \cap J_{1} \right) \cup \left( I_{1} \cap J_{2} \right) \cup \left( I_{2} \cap J_{1} \right) \cup \left( I_{2} \cap J_{2} \right) = \left( I_{1} \cap J_{1} \right) \cup \left( I_{2} \cap J_{2} \right)$.
\end{proof}

\section{Halving opperation}\label{sec:halving}

In this section we describe the \emph{halving operation}. 
We  apply this operation to the automorphism group of a non-degenerate regular polytope $\cP$. 
As result we obtain a group $\halP$, which is a subgroup of $\aut(\cP)$ of index at most $2$.
We prove that the group $\halP$ is a C-group and that the corresponding incidence system is flag-transitive. 
Therefore the group $\halP$ is the type-preserving automorphism group of a regular hypertope.

Let $n \geq 3$ and $\cP$ be a regular, non-degenerate $n$-polytope of type $\{p_{1}, \dots, p_{n-2}, p_{n-1}\}$ and automorphism group $\aut(\cP) = \left\langle \rr_{0}, \dots, \rr_{n-1} \right\rangle $. 
The \emph{halving operation} is the map \[\eta: \left\langle \rr_{0}, \dots, \rr_{n-1} \right\rangle \to \left\langle \rho_{0}, \dots, \rho_{n-1} \right\rangle,  \]
where 
\begin{equation}\label{eq:HalvingGroup}
		\rho_{i} = 
			\begin{cases}
				\rr_{i}, & \text{if}\ 0 \leq i \leq n-2,\\
				\rr_{n-1} \rr_{n-2} \rr_{n-1}, & \text{if}\ i = n-1,
			\end{cases}
	\end{equation}
	
The \emph{halving group} of $\cP$, denoted by $\halP$, is the  image of $\aut(\cP)$ under $\eta$.

Observe that the group $\halP = \langle \rho_{1}, \dots, \rho_{n-1} \rangle $ has the following diagram
	\begin{equation} \label{eq:diagHalving}
		 	\begin{tikzcd}[row sep=tiny, start anchor = center, end anchor = center]
		& & & & & [-0.6em] \overset{\rho_{n-2}}{\bullet} \ar[dd, dash, "s" ]\\
		\overset{\rho_{0}}{\bullet} \ar[r, dash, "p_{1}"' ] &%
		\overset{\rho_{1}}{\bullet} \ar[r, dash, "p_{2}"'] &%
		\overset{\rho_{2}}{\bullet} \ar[r, dash, "{\dots}" description] &%
		\overset{\rho_{n-4}}{\bullet} \ar[r, dash, "p_{n-3}"'] &
		\overset{\rho_{n-3}}{\bullet} \ar[ru, dash, "{p_{n-2}}" near end] \ar[rd, dash, "p_{n-2}"' near end] \\
	 & & & & & [-0.6em] \underset{\rho_{n-1}}{\bullet} 
  \end{tikzcd}
	\end{equation}
where $s = p_{n-1}$ if $p_{n-1}$ is odd, otherwise $s = \frac{p_{n-1}}{2}$.
We denote by $\hypP$ the coset incidence system $\Gamma\left( \halP, \left( H_{i} \right)_{i \in \{0, \dots, n-1\}} \right)$, where $H_{i}$ is the subgroup of $\halP$ generated by $\left\{ \rho_{j} : j \neq i \right\} $. 
In \cref{thm:index2Sbgp} we show that the group $\halP$ satisfies the intersection condition and in \cref{prop:flagTransitivity} we show that the corresponding incidence $\hypP$ is flag-transitive. 
We conclude the section with \cref{coro:HypertopesFromPolytopes} which states that $\hypP$ is in fact a thin, chamber-transitive coset geometry, i.e. a regular hypertope.

The halving operation has been used before in the context of regular polyhedra of type $\left\{ q,4 \right\} $ (see \cite{McMullenSchulte_1997_RegularPolytopesOrdinary} and  \cite[Section 7B]{McMullenSchulte_2002_AbstractRegularPolytopes}) and the resulting incidence system is a regular polyhedron of type $\left\{ q,q \right\} $. 

The operation described above doubles the fundamental region of $\aut(\cP)$ by gluing together the base flag $\Phi$ and the flag $\Phi^{n-1}$. 

As an example we explore the halving operation to the cubic tessellation  $\{4,3,4\}$. The elements of type $0$ and $1$ of the resulting incidence system are the vertices and edges of the $\{4,3,4\}$, respectively.
The elements of type $2$ are half of the cubes and the elements of type $3$ are the other half.
This is the construction of the infinite hypertope described in \cite[Example 2.5]{FernandesLeemansWeiss_ExplorationLocallySpherical_preprint} and can also be seen as a semi-regular polytope (see \cite[Section 3]{MonsonSchulte_2012_SemiregularPolytopesAmalgamated}).

It is easy to see that $\halP$ has index $2$ in $\aut(\cP)$ if and only if the set of facets of $\cP$ is bipartite. 
This is only possible if $p_{n-1}$ is even.
If this is the case, then the elements of type $i$, for $i \in \left\{ 0, \dots, n-3 \right\} $ are the faces of rank $i$ of $\cP$. The elements of type $n-2$ are half of the facets of $\cP$ (those belonging to the same partition as the base facet) and the elements of type $n-1$ are the other half of the facets, namely, those in the same partition as the facet of $\Phi^{n-1}$.

In the remainder of the section we let $\cP$ be a fixed regular $n$-polytope with a base flag $\Phi$, the automorphism group $\aut(\cP) = \left\langle \rr_{0}, \dots \rr_{n-1} \right\rangle $ and $H = \langle \rho_{0}, \dots, \rho_{n-1} \rangle$ the halving group of $\cP$. 
For $i,j \in \left\{ 0, \dots, n-1 \right\} $ we let $H_{i}$  and $H_{i,j}$ be the groups $ \left\langle \rho_{k} : k \neq i \right\rangle $ and $\left\langle \rho_{k} : k \not \in \left\{ i,j \right\}  \right\rangle $, respectively. 
Finally, by $\hypP$ we denote the incidence system $\Gamma\left( H, (H_{i})_{i \in \left\{ 0 ,\dots, n-1 \right\} } \right)$ and by $\Gamma_{i}$ the residue of $\hypP$ induced by $H_{i}$, that is $\Gamma_{i} = \Gamma\left(H_{i}, (H_{i,j})_{j \in \left\{ 0, \dots, n-1 \right\} \sm \left\{ i \right\}  }  \right)$.

\begin{theorem} \label{thm:index2Sbgp} 
Let $n \geq 3$ and $\cP$ be a regular, non-degenerate $n$-polytope of type $\{p_{1}, \dots, p_{n-1}\}$. 
Then the halving group $\halP$ is a C-group.
\end{theorem}
\begin{proof}
	The strategy of this proof is to use \cref{prop:IC}.
	Let $\Phi=\left\{F_{-1}, \dots,  F_{n} \right\} $ be the base flag of $\cP$. 
	Let $F'_{n-1}$ be the facet of $\Phi^{n-1}$. 
	Observe that the groups $H_{n-1}=\left\langle \rho_{0}, \dots, \rho_{n-2} \right\rangle $ and $H_{n-2}=\left\langle \rho_{0},\dots, \rho_{n-3}, \rho_{n-1}  \right\rangle $ are the automorphism groups of the sections $F_{n-1}/F_{-1}$ and $F'_{n-1}/F_{-1} $, respectively.
 Hence, these groups are C-groups (see \cite[Proposition 2B9]{McMullenSchulte_2002_AbstractRegularPolytopes}).
	
	To prove that $H_{i}$ is a C-group for $0 \leq i \leq n-3$ we proceed by induction and use \cref{lemma:CgroupsProduct}. 
	
	If $n=3$, we need to prove that the group $H_{0} = \left\langle \rho_{1}, \rho_{2} \right\rangle = \left\langle \rr_{1}, \rr_{2} \rr_{1} \rr_{1}  \right\rangle$ is a C-group. However, this group is a subgroup of the automorphism group of the polygonal section $F_{3}/F_{1}$ isomorphic to the dihedral group $\mathbb{D}_{s}$. 
	To finish our base case we only need to show
	\begin{align}
		\left\langle \rho_{0}, \rho_{1} \right\rangle \cap \left\langle \rho_{0}, \rho_{2} \right\rangle  &= \left\langle \rho_{0} \right\rangle,\label{eq:Cgroup_rk3_0} \\ 
		\left\langle \rho_{0}, \rho_{1} \right\rangle \cap \left\langle \rho_{1}, \rho_{2} \right\rangle  &= \left\langle \rho_{1} \right\rangle, \label{eq:Cgroup_rk3_1}\\
		\left\langle \rho_{0}, \rho_{2} \right\rangle \cap \left\langle \rho_{1}, \rho_{2} \right\rangle  &= \left\langle \rho_{2} \right\rangle. \label{eq:Cgroup_rk3_2}
	\end{align}
	To prove \eqref{eq:Cgroup_rk3_0}, just observe that $\left\langle \rho_{0} \right\rangle = \stab_{\cP}(\{F_{1}, F_{2}\})$. 
	Let $\gamma \in \left\langle \rho_{0}, \rho_{1} \right\rangle \cap \left\langle \rho_{0}, \rho_{2} \right\rangle$. 
	Since $ \gamma \in \left\langle \rho_{0}, \rho_{1} \right\rangle $, $\gamma$ fixes $F_{2}$. 
	Similarly, since $\gamma \in \left\langle \rho_{0}, \rho_{2} \ \right\rangle $, $\gamma $ must fix $F'_{2}$.
	This implies that $\gamma$ fixes $F_{1}$, since this is the only $1$-face of $\cP$ incident to both $F_{2}$ and $F_{1}$. 
	Therefore, $\gamma \in \stab_{\cP}(\{F_{1}, F_{2}\}) = \left\langle \rho_{0} \right\rangle $. 
	The other inclusion is obvious.
	
	Similarly, we have that $\left\langle \rho_{0}, \rho_{1} \right\rangle \cap \left\langle \rho_{1}, \rho_{2} \right\rangle \subset \stab_{\cP}(\{F_{0}, F_{1}\})$. 
	This follows from the fact that the group $\left\langle \rho_{0}, \rho_{1} \right\rangle $ fixes $F_{2}$ and the group $\left\langle \rho_{1}, \rho_{2} \right\rangle $ fixes $F_{0}$. 
	Then, $\left\langle \rho_{0}, \rho_{1} \right\rangle \cap \left\langle \rho_{1}, \rho_{2} \right\rangle \subset \left\langle \rho_{1} \right\rangle $. 
	Again, the other inclusion is obvious. 
	The proof of \eqref{eq:Cgroup_rk3_1} follows from the same argument but now with respect to the flag $\Phi^{2}=\left\{ F_{0}, F_{1}, F'_{2} \right\} $ of $\cP$. 	This completes the base case. 

	Assume that the halving group $\halP[\cF]$ of every non-degenerate regular polytope $\cF$ of rank $r$ with $3 \leq r < n$ is a C-group. 
	Observe that if $i \in \{0, 1, \dots, n-3\}$ then $H_{i} = H_{i}^{-} \times H_{i}^{+}$ where $H_{i}^{-} = \left\langle \rho_{j} :j < i \right\rangle $ and $H_{i}^{+} = \left\langle \rho_{j} : i < j \right\rangle $. 
	Note that $H_{i}^{-}$ is just the automorphism of the section $F_{i}/F_{-1}$. If $i < n-3$ then $H_{i}^{+}$ is the halving group of the section $F_{n}/F_{i}$, which is a C-group  by the inductive hypothesis. 
	If $i= n-3$, then $H_{i}^{+}$ is isomorphic to a dihedral group $\mathbb{D}_{s}$. 
	In any case, it follows from \cref{lemma:CgroupsProduct} that $H_{i}$ is a C-group.
	
	In order to use \cref{prop:IC}, we need to prove that for every pair $i,j \in \{0, \dots, n-1\}$, with $i < j $, the equality 
	\begin{equation}\label{eq:Cgroup_rkn}
		H_{i} \cap H_{j} = H_{i,j}
	\end{equation}
 holds. 
 We proceed in a similar way as in rank $3$. 
 If $\{i,j\} = \left\{ n-1,n-2 \right\}$, then observe that $H_{i}=H_{n-2}$ fixes $F_{n-1}$ and $H_{j}=H_{n-1}$ fixes $F'_{n-1}$.
This implies that an element $\gamma \in H_{n-2} \cap H_{n-1}$ must fix $F_{n-2}$. Thus $\gamma \in \stab_{\cP}(\left\{F_{n-2}, F_{n-1} \right\} ) = \left\langle \rho_{0}, \dots, \rho_{n-3} \right\rangle=H_{i,j}$. 
The other inclusion is obvious.

If $j \in \{n-1,n-2\}$ and $i \leq n-3$ then \eqref{eq:Cgroup_rkn} follows from the fact that $H_{j}$ is a string C-group.

Assume that $0\leq i < j \leq n-3$. 
Let $\cF$ be the section $F_{j}/F_{-1} = $ of $\cP$. 
Let $\gamma \in H_{i} \cap H_{j}$. 
Observe that $H_{j} = H^{-}_{j} \times H_{j}^{+}$ and that $\aut(\cF)= H_{j}^{-}$. Let $\alpha \in H_{j}^{-}$ and $\beta \in H_{j}^{+}$ be such that $\gamma = \alpha \beta$. 
Note that $\beta$ fixes the face $F_{i}$ of $\cP$ and since $\gamma \in H_{i}$, then $\alpha$ must fix $F_{i}$. 
Since $H_{j}^{-}$ is a string C-group, it follows that $\alpha \in \left\langle \rho_{0}, \dots, \rho_{i-1}, \rho_{i+1}, \dots, \rho_{j-1} \right\rangle$. 
Then \[\gamma \in \left\langle \rho_{0}, \dots, \rho_{i-1}, \rho_{i+1}, \dots, \rho_{j-1} \right\rangle \times \left\langle \rho_{j+1}, \dots, \rho_{n-1} \right\rangle = H_{i,j}.\] 
The other inclusion is obvious.
\end{proof}

The halving group of a regular polytope $\cP$ is a C-group with Coxeter diagram \eqref{eq:diagHalving}. 
Groups generated by involutions with this diagram are called \emph{tail-triangle groups} (even when $s = 2$).
In \cite{MonsonSchulte_2012_SemiregularPolytopesAmalgamated, MonsonSchulte_2019_AssemblyProblemAlternating} Monson and Schulte show that when a tail-triangle group is a C-group, it is the automorphism group of an alternating semi-regular polytope. 
We denote by $\semi(\cP)$ the semi-regular polytope obtained by the halving operation on $\cP$.
The polytope $\semi(\cP)$ has two orbits of isomorphic regular facets, namely the right cosets of $\langle \rho_{0}, \dots, \rho_{n-3}, \rho_{n-2} \rangle$ and $\langle \rho_{0}, \dots, \rho_{n-3}, \rho_{n-1} \rangle$.
These \emph{base facets} are incident with a regular polytope $\mathcal{R}$ of rank $n-2$.
In fact, every flag of $\mathcal{R}$ can be extended to a flag of $\semi(\cP)$ in two different ways.
Moreover, any flag of $\semi(\cP)$ belongs to the orbit of one of these two flags.
The automorphisms of $\semi(\cP)$ can now be used to show flag transitivity of $\hypP$.

\begin{proposition} \label{prop:flagTransitivity}
 The incidence system $\hypP$ associated with a non-degenerate regular polytope $\cP$ is flag-transitive.
\end{proposition}
\begin{proof}
 Let $J \subset \left\{ 0, \dots, n-1 \right\} $ and let $F=\{H_{i} h : i \in J\}$ for some $h \in \halP$ be a flag of $\hypP$ of type $J$.  
 If $|J \cap \{n-2,n-1\}| \leq 1$, then $F$ is a chain of $\semi(\cP)$ of type $J$. By \cite[Lemma 4.5b]{MonsonSchulte_2012_SemiregularPolytopesAmalgamated}, the group $\halP$ is transitive on the chains of this type.
 
 If $\{n-2, n-1\} \subset J$, then $F= \Upsilon \cup \Upsilon^{n-1}$, where $\Upsilon$ is the chain of $\semi(\cP)$ of type $J'=J \sm \{n-1\}$ whose faces are contained in $F$.
 Again, \cite[Lemma 4.5b]{MonsonSchulte_2012_SemiregularPolytopesAmalgamated} implies that $\halP$ is transitive in chains of type $J'$. 
 Finally, observe that if $\beta \in \halP$, then $(\Upsilon \beta)^{n-1} = (\Upsilon^{n-1})\beta$.
 It follows that $\halP$ is also transitive on flags of type $J$. 
\end{proof}

\begin{corollary}\label{coro:HypertopesFromPolytopes}
 Let $\cP$ be a non-degenerate, regular $n$-polytope and $I = \left\{ 0, \dots, n-1 \right\} $. Let $\halP$ be the halving group of $\cP$. Then the incidence system $\hypP = \Gamma\left( \halP, (H_{i})_{i \in I } \right)$ is a regular hypertope such that $\autI\left(\hypP \right)= \halP$.
\end{corollary}

The assumption that $\cP$ is non-degenerate is very important. When the halving operation is applied on the dual of the $4$-hemicube the resulting incidence system is not a hypertope (see \cite[Example 3.3]{FernandesLeemansWeiss_ExplorationLocallySpherical_preprint}).

Theorem 2.5 of \cite{MonsonSchulte_2019_AssemblyProblemAlternating} implies  that the semi-regular polytope $\semi(\cP)$ is regular because the associated C-group admits a group automorphism interchanging the generators $\rho_{n-1}$ and $\rho_{n-2}$ (this automorphism is given by conjugation by $\rr_{n-1}$). 
The polytope $\semi(\cP)$ is in fact isomorphic to $\cP$.   

\section{Locally spherical hypertopes from generalised cubes.} \label{sec:lshGeneralisedCubes}

A \emph{spherical hypertope} is a universal hypertope 
whose type-preserving automorphism group is finite. 
A \emph{locally spherical hypertope} is a regular hypertope whose all proper residues are spherical hypertopes.
A regular hypertope is of \emph{eucliden type} if the type-preserving automorphism group of its regular cover is an affine Coxeter group. 
Following \cite{FernandesLeemansWeiss_ExplorationLocallySpherical_preprint}, a locally spherical hypertope is of  \emph{hyperbolic type} if the type-preserving automorphism group  of its universal cover is a compact hyperbolic Coxeter group.
It is well known that compact hyperbolic Coxeter groups exist only in ranks $3$, $4$ and $5$.

In \cite{FernandesLeemansWeiss_ExplorationLocallySpherical_preprint} the authors show that  a locally spherical hypertope has to be of spherical, euclidean or hyperbolic type. 
The complete list of the Coxeter diagrams of these groups can be found in \cite[Tables 1 and 2]{FernandesLeemansWeiss_ExplorationLocallySpherical_preprint}.
Whereas the first two classes are well understood, not much is known about hyperbolic type.
In particular, the authors were not successful in producing any finite example of rank $5$ locally spherical proper hypertope of this type. 
In what follows we will use the halving operation on a certain class of polytopes first described by Danzer in \cite{Danzer_1984_RegularIncidenceComplexes} (see also \cite[Theorem 8D2]{McMullenSchulte_2002_AbstractRegularPolytopes}) and in particular we produce an example of finite rank $5$ proper regular hypertope of hyperbolic type. 

We briefly review Danzer's construction of generalised cubes.

Let $\cK$ be a regular finite non-degenerate $n$-polytope with vertex set $V=\left\{ v_{1}, \dots, v_{m} \right\} $.  
Consider the set \[2^{V} = \prod_{j=1}^{m}\left\{ 0,1 \right\} = \left\{ \vx = (x_{1}, \dots, x_{m}) : x_{j} \in \left\{ 0,1 \right\}  \right\}.  \]

Given an $i$-face $F$ of $\cK$ and $\vx \in 2^{V}$ define the sets \[F(\vx) =\left\{ \vy=(y_{1}, \dots, y_{m}) \in 2^{V} : y_{j} = x_{j} \text{ if } v_{j} \not\leq F\right\}. \] 
Then the polytope $\twoK$ is the partially ordered set \[\{\emptyset\}\cup 2^{V} \cup \left\{ F(\vx) : F \in \cK, \vx \in 2^{V} \right\} \] ordered by inclusion.
The improper face of rank $-1$ of $\twoK$ is $\emptyset$ and if $i \geq 0$ the $i$-faces of $\twoK$ are the sets $F(\vx)$ for $F$ a certain $(i-1)$-face of $\cK$ and some $\vx \in 2^{V}$.

If $\cK$ is a regular of type $\left\{ p_{1}, \dots, p_{n-2} \right\} $ then $\twoK$ is a regular polytope of type $\left\{ 4, p_{1}, \dots, p_{n-2} \right\}$. 
In fact, all the vertex figures of $\twoK$ are isomorphic to $\cK$. 
The polytope $\twoK$ is called a \emph{generalised cube} since when $\cK$ is the $(n-1)$-simplex, the polytope $\twoK$ is isomorphic to the $n$-cube.

For our purposes it is convenient to denote by $\dtwoK$  the polytope $\left(  \twoK[\cK^{\ast}]\right)^{\ast}$, so that $\dtwoK$ is a regular polytope of type $\left\{ p_{1}, \dots, p_{n-2}, 4 \right\} $ whose facets are isomorphic to $\cK$.

The automorphism group of $\dtwoK$ is isomorphic to $\bZ_{2}^{m} \rtimes \aut(\cK)$, where $m$ denotes the number of facets of $\cK$ and the action of $\aut(\cK)$ on $\bZ_{2}^{m}$ is given by permuting coordinates in the natural way.
In particular, the size of this group is $2^{m} \times |\aut(\cK)|$ (see \cite[Theorem 2C5]{McMullenSchulte_2002_AbstractRegularPolytopes} and \cite{Pellicer_2009_ExtensionsRegularPolytopes}).  

\begin{rem}
	In \cite{Pellicer_2009_ExtensionsRegularPolytopes} Pellicer generalises Danzer's construction of $\twoK$. 
	Given a finite non-degenerate regular $(n-1)$-polytope $\cK$ of type $\left\{ p_{1}, \dots, p_{n-2} \right\} $ and $s \in \bN$, Pellicer's construction gives as a result an $n$-polytope $\cP_{s}$ of type $\left\{p_{1}, \dots p_{n-2}, 2s \right\} $.
	If $s = 2$, the polytope $\cP_{2}$ is isomorphic to $\dtwoK$.
	However, when our construction is applied to the polytopes $\cP_{s}$ for $s \geq 3$, the resulting hypertopes are not locally spherical and therefore not included that in this paper.
\end{rem}

Now we discuss the locally spherical hypertopes resulting from applying halving operation to the polytopes obtained from Danzer's construction. Since $\dtwoK$ is of type $\{p_{1}, \dots p_{n-2}, 4\}$, the hypertope $\hypP[\dtwoK]$ has the following Coxeter diagram:
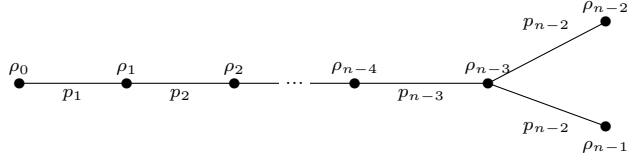
\[
	 	\begin{tikzcd}[row sep=tiny, start anchor = center, end anchor = center]
		& & & & & [-0.6em] \overset{\rho_{n-2}}{\bullet}\\
		\overset{\rho_{0}}{\bullet} \ar[r, dash, "p_{1}"' ] &%
		\overset{\rho_{1}}{\bullet} \ar[r, dash, "p_{2}"'] &%
		\overset{\rho_{2}}{\bullet} \ar[r, dash, "{\dots}" description] &%
		\overset{\rho_{n-4}}{\bullet} \ar[r, dash, "p_{n-3}"'] &
		\overset{\rho_{n-3}}{\bullet} \ar[ru, dash, "{p_{n-2}}" near end] \ar[rd, dash, "p_{n-2}"' near end] \\
	 & & & & & [-0.6em] \underset{\rho_{n-1}}{\bullet} 
  \end{tikzcd}
\]
We naturally extend the Schläfli symbol and say that $\hypP[\dtwoK]$ is of type $\{p_{n-1}, \dots, p_{n-3}, {}_{p_{n-2}}^{p_{n-2}}\}$.

In rank $3$ the polytope $\dtwoK[\{p\}]$ is obtained by applying the construction on a regular polygon $\{p\}$ and the induced hypertope is in fact a self-dual polyhedron of type $\left\{ p,p \right\} $. 
This polyhedron  has $2^{p-1}$ vertices, $2^{p-2}p$ edges and $2^{p-1}$ faces an it is a map on a surface of genus $2^{p-3}(p-4) + 1$. 
For $p = 3$ the resulting hypertope is a spherical polyhedron $\left\{ 3,3 \right\} $, i.e. the tetrahedron. 
When $p = 4$ the polytope $\dtwoK[\left\{ 4 \right\} ]$ is the toroid $\left\{ 4,4 \right\}_{(4,0)} $ and the induced hypertope is also of euclidean type, more precisely it is the toroid $\left\{ 4,4 \right\}_{(2,2)} $. 
 
To obtain locally spherical hypertopes in rank $4$, $\cK$ must be of type $\left\{ p,3 \right\} $ with $p=3,4,5$. The resulting hypertopes are of spherical, euclidean, and hyperbolic type, respectively.
If $p=3$, the hypertope $\hypP[\dtwoK]$ is the universal hypertope of Coxeter diagram $D_{4}$ and type $\{3, {}_{3}^{3}\}$. 
When $p=4$ the polytope $\dtwoK$ is the toroid $\left\{ 4,3,4 \right\}_{(4,0,0)} $ and $\hypP[\dtwoK]$ is a toroidal hypertope described by Ens in \cite[Theorem 4.3]{Ens_2018_Rank4Toroidal}, which we denote by $\left\{ 4,{}_{3}^{3} \right\}_{(4,0,0)} $. 
The automorphism group of this hypertope has Coxeter diagram $\tilde{B}_{3}$. 
 If $p = 5$, the resulting hypertope is of type $\left\{ 5,{}_{3}^{3} \right\} $  with automorphism group of size $2^{11} \times 120  = 245,760$. This example is different from any in \cite{FernandesLeemansWeiss_ExplorationLocallySpherical_preprint}.

In rank $5$ the polytope $\cK$ must be of type $\left\{ p, 3, 3 \right\} $ with $p = 3,4,5$, the resulting hypertopes are of shperical, Euclidean and hyperbolic type, respectively. 
If $p=3$ then the hypertopes is the universal spherical hypertope of type $\left\{ 3, 3, {}_{3}^{3} \right\} $, i.e. the universal hypertope of Coxeter diagram $D_{5}$. 
If $p=4$ the polytope $\dtwoK$ is the regular toroid $\left\{ 4,3,3,4 \right\}_{(4,0,0,0)} $. 
The induced hypertope is of Euclidean type, hence a toroidal hypertope which we denote by $\left\{ 4,3,{}_{3}^{3} \right\}_{(4,0,0,0)} $. 
The Coxeter diagram of its automorphism group is $\tilde{B}_{4}$.
For $p = 5$ the regular polytope $\dtwoK$ is constructed from the $120$-cell. The hypertope $\hypP[\dtwoK]$ is of type $\left\{ 5,3, {}_{3}^{3} \right\} $ and its automorphism group has size $2^{119} \times 14400$. 
It is not surprising that the authors of \cite{FernandesLeemansWeiss_ExplorationLocallySpherical_preprint}  could not find this example using a computational approach. 

For rank $n\geq 6$ we can only obtain locally spherical hypertopes from our construction if $\cK$ is the $(n-1)$-simplex $\{3^{n-2}\}$ or the $(n-1)$-cube $\{4, 3^{n-3}\}$. 
The polytope $\dtwoK$ is the $n$-cross-polytope $\{3^{n-2}, 4\}$ or the toroid $\{4, 3^{n-3}, 4\}_{(4,0, \dots 0)}$, respectivelly.
In the former case the resulting hypertope is the universal spherical hypertope of type $\{3^{n-3}, {}_{3}^{3} \}$ associated with the Coxeter diagram $D_{n}$ while in the latter it is a toroidal hypertope associated with the Coxeter diagram $\tilde{B}_{n-1}$ which we denote by $\{4, 3^{n-4}, {}_{3}^{3}\}_{(4,0 ,\dots 0)}$.

\bibliographystyle{plainurl}
\bibliography{danzerLSHypertopes.bib}

\end{document}